\documentclass[letterpaper,10pt]{article}

\usepackage[margin=0.5in]{geometry}

\usepackage{fullpage}
\usepackage{enumerate}
\usepackage{authblk}
\usepackage{hyperref}
\usepackage{verbatim}
\usepackage{section, amsthm, textcase, setspace, amssymb, amscd, lineno, amsmath, amsfonts, latexsym, fancyhdr, longtable, ulem, mathtools}
\usepackage{epsfig, graphicx, pstricks,pst-grad,pst-text,tikz,colortbl}
\usepackage{epsf}
\usepackage{graphicx, color}
\usepackage{float}
\usepackage[rflt]{floatflt}
\usepackage{amsfonts}
\usepackage{latexsym,enumitem}
\usetikzlibrary{fit,matrix,positioning}
\usepackage{pdflscape}
\usetikzlibrary{decorations.pathreplacing}
\usepackage[most]{tcolorbox}
\usepackage{lipsum}
\usetikzlibrary{decorations.markings}
\usepackage[most]{tcolorbox}

\usepackage{mathrsfs}

\newcommand{\ldownarrow}{\Big\downarrow}

\newcommand{\precdot}{\prec\mathrel{\mkern-5mu}\mathrel{\cdot}}

\newtheorem{theorem}{Theorem}[section]
\newtheorem{lemma}[theorem]{Lemma}

\newtheorem{example}[theorem]{Example}

\newtheorem{prop}[theorem]{Proposition}
\newtheorem*{theorem*}{Theorem}

\def\celx(#1,#2)[#3]
{
  \node at (#1+0.5, #2+0.5) {$#3$};
  \draw (#1,#2) rectangle (#1+1, #2+1);
}

\DeclareMathOperator{\MC}{MC}
\DeclareMathOperator{\room}{room}
\DeclareMathOperator{\blockers}{blockers}

\usepackage{xargs}
\usepackage[colorinlistoftodos,prependcaption,textsize=tiny,linecolor=red,backgroundcolor=red!25,bordercolor=red]{todonotes}
\usepackage[backend=bibtex]{biblatex}
\begin{document}

\title{Two combinatorial puzzles arising from the theory of Kohnert polynomials}

\author[*]{Theo Koss}

\author[**]{Nicholas Mayers}

\author[***]{Alex Moon}

\affil[*]{Department of Mathematics, University of Wisconsin-Milwaukee, Milwaukee, WI, 53211}

\affil[**]{Department of Mathematics, Kennesaw State University, Kennesaw, GA, 30144}

\affil[***]{Department of Mathematics, Dartmouth College, Hanover, NH, 03755}

\maketitle

\bigskip
\begin{abstract} 
\noindent
Motivated by recent work of Hanser and Mayers, we study two combinatorial puzzles arising from the theory of Kohnert polynomials. Such polynomials are defined as generating polynomials for certain collections of diagrams consisting of unit cells arranged in the first quadrant generated from an initial ``seed diagram" by applying what are called ``Kohnert moves". Each Kohnert move affects the position of at most cell of a diagram, attempting to move the rightmost cell of a given row to the first available position below and in the same column. In this paper, we study the combinatorial puzzles defined as follows: given a 
diagram $D$, form a diagram that is fixed by all Kohnert moves by applying either the fewest or most possible number of Kohnert moves. For both puzzles, we find complete solutions as well as methods for combinatorially computing the associated number of Kohnert moves in terms of the initial diagram $D$. 
\end{abstract}


\section{Introduction}

In recent work \cite{Puzzle1}, Hanser and Mayers introduce a combinatorial puzzle which arises from the theory of Lascoux polynomials \cite{Lascoux}. For their puzzle, one is given an arbitrary diagram $D$ consisting of a finite number of unit cells arranged into rows and columns in the first quadrant. From such a diagram, one can form other diagrams by applying what are called ``Kohnert" and ``ghost moves". Both such moves cause at most one cell of a diagram to move to a position in the same column and a lower row, with ghost moves also creating special ``ghost cells". Starting from the given diagram $D$, the goal of their puzzle is to use Kohnert and ghost moves to produce a diagram with the maximum possible number of ghost cells.\footnote{For the connection to Lascoux polynomials, it was shown in \cite{Pan1} that for special choices of initial diagrams -- called ``key diagrams" -- the collection of diagrams formed by applying sequences of ghost and Kohnert moves can be weighted in such a way that their sum is a given Lascoux polynomial. Restricting to just Kohnert moves and weighting diagrams as in the case of Lascoux polynomials, one is led to the family of Kohnert polynomials \cite{KP3,KP2,KP1} which contains both Schubert \cite{AssafSchu,LS82a} and key polynomials \cite{Dem74a,Dem74b,Kohnert}.} Here, we consider two combinatorial puzzles which arise when one applies only Kohnert moves to a given initial diagram, providing complete solutions to both.


In order to describe the combinatorial puzzles of interest, first we make precise the effect of applying a Kohnert move. Given a diagram $D$, applying a Kohnert move at row $r\in\mathbb{Z}_{>0}$ causes the rightmost cell in row $r$ to move to the first empty position below and in the same column, provided that such a position exists in the first quadrant. For example, in Figure~\ref{fig:intro} (a) we illustrate a diagram $D$ and in Figure~\ref{fig:intro} (b) the diagram formed from $D$ by applying a Kohnert move at row 3.

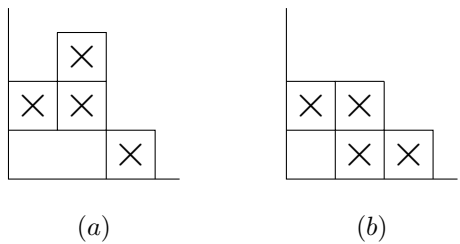
\begin{figure}[H]
    \centering
    $$\begin{tikzpicture}[scale=0.65]
  \node at (1.5, 2.5) {$\bigtimes$};
  \node at (0.5, 1.5) {$\bigtimes$};
  \node at (1.5, 1.5) {$\bigtimes$};
  \node at (2.5, 0.5) {$\bigtimes$};
  \node at (1.75,-1) {$(a)$};
  \draw (0,3.5)--(0,0)--(3.5,0);
  \draw (1,3)--(2,3)--(2,2)--(1,2)--(1,3);
  \draw (2,2)--(2,1)--(1,1)--(1,2)--(0,2);
  \draw (0,1)--(1,1);
  \draw (2,0)--(2,1)--(3,1)--(3,0);
\end{tikzpicture}\quad\quad\quad\quad \begin{tikzpicture}[scale=0.65]
  \node at (1.5, 0.5) {$\bigtimes$};
  \node at (0.5, 1.5) {$\bigtimes$};
  \node at (1.5, 1.5) {$\bigtimes$};
  \node at (2.5, 0.5) {$\bigtimes$};
  \node at (1.75,-1) {$(b)$};
  \draw (0,3.5)--(0,0)--(3.5,0);
  \draw (2,2)--(1,2);
  \draw (2,2)--(2,1)--(1,1)--(1,2)--(0,2);
  \draw (1,0)--(1,1);
  \draw (0,1)--(1,1);
  \draw (2,0)--(2,1)--(3,1)--(3,0);
\end{tikzpicture}$$
    \caption{Diagrams and Kohnert moves}
    \label{fig:intro}
\end{figure}

\noindent
Now, starting with a diagram $D$ and applying Kohnert moves, within a finite number of moves one will form a diagram $\widehat{D}$ which is fixed by all Kohnert moves. We refer to such a diagram $\widehat{D}$ that is fixed by all Kohnert moves as a \textbf{minimal diagram} due to a connection to an associated poset, called a \textbf{Kohnert poset} (see Section~\ref{sec:prelim}). For an example, the diagram illustrated in Figure~\ref{fig:intro} (b) is minimal. With this, the combinatorial puzzles which concern us here are defined as follows. Given an arbitrary diagram $D$, form a minimal diagram using either as few or as many Kohnert moves as possible. As noted above, we provide complete solutions to both puzzles. In particular, for both puzzles we
\begin{enumerate}
    \item[(1)] describe sequences of Kohnert moves which correspond to a solution (see the ends of Sections~\ref{sec:max} and~\ref{sec:min}) and
    \item[(2)] provide combinatorial means of computing the associated number of Kohnert moves in a solution in terms of only the initial given diagram $D$ (see Theorems~\ref{thm:max} and~\ref{thm:min}).
\end{enumerate}
It is worth noting that in the case of the puzzle corresponding to using the maximum number of Kohnert moves, the combinatorial value computed in (2) corresponds to the length of the associated Kohnert poset.

The remainder of the paper is organized as follows. In Section~\ref{sec:prelim}, we set the definitions and notation required to study our combinatorial puzzles. Following this, in Sections~\ref{sec:max} and~\ref{sec:min}, we establish our main results. 


\section{Preliminaries}\label{sec:prelim}

As mentioned in the introduction, our puzzles involve diagrams and certain moves on diagrams. Formally, a \textbf{diagram} $D$ is an array of finitely many cells in $\mathbb{Z}_{>0}\times\mathbb{Z}_{>0}$ (see Figure~\ref{fig:diagram}). We may also think of a diagram as the set of row-column coordinates of the cells defining it, where rows are labeled from bottom to top and columns from left to right. Therefore, we write $(r,c)\in D$ if the cell in row $r$ and column $c$ is in the diagram $D$; otherwise, if $D$ contains no cell in position $(r,c)$, then we write $(r,c)\notin D$. We make precise the intuitive idea of an ``empty'' position in a diagram as follows. An \textbf{empty position} of $D$ is defined as a position containing no cell in $D$, but which lies below at least one cell of $D$. We denote the set of all such positions with $\mathrm{empty}(D)$, i.e.,
$$\mathrm{empty}(D):=\{(r,c)\notin D~|~\exists \tilde{r}>r~\text{such that}~(\tilde{r},c)\in D\}.$$ In addition to the notation introduced above, if the topmost nonempty row of $D$ is row $m$, then we define the \textbf{row weight} of $D$ by $\mathrm{rwt}(D):=(r_1,r_2,\hdots,r_m)$ where $r_i$ is the number of cells in row $i$ of $D$ for $1\le i\le m$. Similarly, if the rightmost nonempty column of $D$ is $n$, then we define the \textbf{column weight} of $D$ by $\mathrm{cwt}(D):=(c_1,c_2,\hdots,c_n)$ where $c_i$ is the number of cells in column $i$ of $D$ for $1\le i\le n$.

\begin{example}
In Figure~\ref{fig:diagram}, we illustrate the diagram $D_0=\{(1,1),(2,1),(2,2),(2,3),(3,3)\}$. For $D_0$, we have $\mathrm{empty}(D_0)=\{(1,2),(1,3)\}$, $\mathrm{rwt}(D_0)=(1,3,1)$, and $\mathrm{cwt}(D_0)=(2,1,2)$.

\begin{figure}[H]
    \centering
    $$\begin{tikzpicture}[scale=0.65]
  \node at (0.5, 0.5) {$\bigtimes$};
  \node at (0.5, 1.5) {$\bigtimes$};
  \node at (1.5, 1.5) {$\bigtimes$};
  \node at (2.5, 2.5) {$\bigtimes$};
  \node at (2.5, 1.5) {$\bigtimes$};
  \draw (0,3.5)--(0,0)--(3.5,0);
  \draw (1,0)--(1,2)--(0,2);
  \draw (0,1)--(1,1);
  \draw (1,1)--(2,1)--(2,2)--(1,2)--(1,1);
  \draw (2,1)--(3,1)--(3,2)--(2,2)--(2,1);
  \draw (2,2)--(3,2)--(3,3)--(2,3)--(2,2);
\end{tikzpicture}$$
    \caption{Diagram $D_0=\{(1,1),(2,1),(2,2),(2,3),(3,3)\}$}
    \label{fig:diagram}
\end{figure}
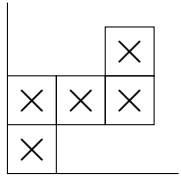
\end{example}

Recall from the introduction that given a diagram $D$ and $r\in\mathbb{Z}_{>0}$, applying a \textbf{Kohnert move} at row $r$ of $D$ results in moving the rightmost cell in row $r$ of $D$ to the first \textit{empty} position below and in the same column (if such a position exists). If applying a Kohnert move at row $r$ of $D$ causes the cell in position $(r,c)\in D$ to move down to position $(r',c)$, forming the diagram $D'$, then we write 
$$D'=\mathcal{K}(D,r)=D\ldownarrow^{(r,c)}_{(r',c)}$$ 
and refer to this Kohnert move as moving the cell $(r,c)$ to position $(r',c)$. If there is no empty position below and in the same column as the rightmost cell $(r,c)$ in row $r$ of $D$, then applying a Kohnert move at row $r$ of $D$ does nothing, i.e., $\mathcal{K}(D,r)=D$. Kohnert moves that result in a new diagram are called \textbf{nontrivial}, while those that do not are called \textbf{trivial}. We let $\mathrm{KD}(D)$ denote the set of all diagrams that can be obtained from $D$ by applying a (possibly empty) sequence of Kohnert moves. Note that although some diagrams can be formed in multiple ways from an initial diagram $D$, the collection $\mathrm{KD}(D)$ is a set and does not contain repeated elements. Among the diagrams of $\mathrm{KD}(D)$, we denote the collection of those for which all Kohnert moves are trivial by $\textrm{Min}(D)$, i.e., $$\mathrm{Min}(D):=\{\tilde{D}\in \mathrm{KD}(D)~|~\mathcal{K}(\tilde{D},r)=\tilde{D}~\forall r\in\mathbb{Z}_{>0}\}.$$ The diagrams of $\mathrm{Min}(D)$ are referred to as \textbf{minimal diagrams}. For an example of $\mathrm{KD}(D)$ and $\mathrm{Min}(D)$, see Example~\ref{ex:KDD}.

We now define our combinatorial puzzles of interest. Given a diagram $D$, form a diagram belonging to $\mathrm{Min}(D)$ using either
\begin{enumerate}
    \item[$(1)$] the largest number of nontrivial Kohnert moves possible or
    \item[$(2)$] the least number of nontrivial Kohnert moves possible.
\end{enumerate}
In addition to finding solutions to these puzzles, we also establish combinatorial means for determining when one has solved such a puzzle; that is, intuitively, given a diagram $D$, we establish combinatorial means of computing
\begin{enumerate}
    \item[$(1)$] the maximum number of nontrivial Kohnert moves one can apply to $D$ to form a diagram of $\mathrm{Min}(D)$, denoted $\mathrm{MC}(D)$ and
    \item[$(2)$] the minimum number of nontrivial Kohnert moves one needs to apply to $D$ to form a diagram of $\mathrm{Min}(D)$, denoted $\mathrm{mc}(D)$.
\end{enumerate}
To provide more formal definitions of $\mathrm{MC}(D)$ and $\mathrm{mc}(D)$, we set the following notation. For $\tilde{D}\in\mathrm{KD}(D)$, let $\mathrm{KSeq}(D,\tilde{D})$ denote the set of all sequences of rows $(r_1,\hdots,r_n)\in\mathbb{Z}_{>0}^n$ for which setting $D_0=D$ and $D_i=\mathcal{K}(D_{i-1},r_i)$ for $i\in[n]$, we have $D_{i-1}\neq D_i$ for $i\in [n]$ and $D_n=\tilde{D}$; that is, $\mathrm{KSeq}(D,\tilde{D})$ is the set of all sequences of rows for which when the associated nontrivial Kohnert moves are applied to $D$ in sequence, the resulting diagram is $\tilde{D}$. Using this notation, we have $$\mathrm{MC}(D)=\max_{\tilde{D}\in\mathrm{Min}(D)}\{n~|~(r_1,\hdots,r_n)\in \mathrm{KSeq}(D,\tilde{D})\}$$ and $$\mathrm{mc}(D)=\min_{\tilde{D}\in\mathrm{Min}(D)}\{n~|~(r_1,\hdots,r_n)\in \mathrm{KSeq}(D,\tilde{D})\}.$$

\begin{example}\label{ex:KDD}
In Figure~\ref{fig:KD}, we illustrate the diagrams of $KD(D_0)$ for the diagram $D_0$ of Figure~\ref{fig:diagram}. For $D_0$, we have $\mathrm{Min}(D_0)=\{D_2,D_4\}$, $\mathrm{KSeq}(D_0,D_2)=\{(3),(2,3)\}$, and $\mathrm{KSeq}(D_0,D_4)=\{(2,2,3)\}$. Consequently, $\mathrm{MC}(D_0)=3$ and $\mathrm{mc}(D_0)=1$.
    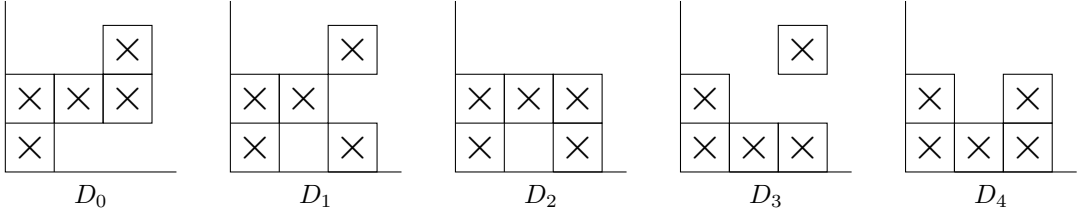
\begin{figure}[H]
    \centering
    $$\begin{tikzpicture}[scale=0.65]
  \node at (0.5, 0.5) {$\bigtimes$};
  \node at (0.5, 1.5) {$\bigtimes$};
  \node at (1.5, 1.5) {$\bigtimes$};
  \node at (2.5, 2.5) {$\bigtimes$};
  \node at (2.5, 1.5) {$\bigtimes$};
  \draw (0,3.5)--(0,0)--(3.5,0);
  \draw (1,0)--(1,2)--(0,2);
  \draw (0,1)--(1,1);
  \draw (1,1)--(2,1)--(2,2)--(1,2)--(1,1);
  \draw (2,1)--(3,1)--(3,2)--(2,2)--(2,1);
  \draw (2,2)--(3,2)--(3,3)--(2,3)--(2,2);
  \node at (1.75,-0.5) {$D_0$};
\end{tikzpicture}\quad\quad\begin{tikzpicture}[scale=0.65]
  \node at (0.5, 0.5) {$\bigtimes$};
  \node at (0.5, 1.5) {$\bigtimes$};
  \node at (1.5, 1.5) {$\bigtimes$};
  \node at (2.5, 2.5) {$\bigtimes$};
  \node at (2.5, 0.5) {$\bigtimes$};
  \draw (0,3.5)--(0,0)--(3.5,0);
  \draw (1,0)--(1,2)--(0,2);
  \draw (0,1)--(1,1);
  \draw (1,1)--(2,1)--(2,2)--(1,2)--(1,1);
  \draw (2,0)--(3,0)--(3,1)--(2,1)--(2,0);
  \draw (2,2)--(3,2)--(3,3)--(2,3)--(2,2);
  \node at (1.75,-0.5) {$D_1$};
\end{tikzpicture}\quad\quad\begin{tikzpicture}[scale=0.65]
  \node at (0.5, 0.5) {$\bigtimes$};
  \node at (0.5, 1.5) {$\bigtimes$};
  \node at (1.5, 1.5) {$\bigtimes$};
  \node at (2.5, 1.5) {$\bigtimes$};
  \node at (2.5, 0.5) {$\bigtimes$};
  \draw (0,3.5)--(0,0)--(3.5,0);
  \draw (1,0)--(1,2)--(0,2);
  \draw (0,1)--(1,1);
  \draw (1,1)--(2,1)--(2,2)--(1,2)--(1,1);
  \draw (2,0)--(3,0)--(3,1)--(2,1)--(2,0);
  \draw (2,1)--(3,1)--(3,2)--(2,2)--(2,1);
  \node at (1.75,-0.5) {$D_2$};
\end{tikzpicture}\quad\quad\begin{tikzpicture}[scale=0.65]
  \node at (0.5, 0.5) {$\bigtimes$};
  \node at (0.5, 1.5) {$\bigtimes$};
  \node at (1.5, 0.5) {$\bigtimes$};
  \node at (2.5, 2.5) {$\bigtimes$};
  \node at (2.5, 0.5) {$\bigtimes$};
  \draw (0,3.5)--(0,0)--(3.5,0);
  \draw (1,0)--(1,2)--(0,2);
  \draw (0,1)--(1,1);
  \draw (1,0)--(2,0)--(2,1)--(1,1)--(1,0);
  \draw (2,0)--(3,0)--(3,1)--(2,1)--(2,0);
  \draw (2,2)--(3,2)--(3,3)--(2,3)--(2,2);
  \node at (1.75,-0.5) {$D_3$};
\end{tikzpicture}\quad\quad\begin{tikzpicture}[scale=0.65]
  \node at (0.5, 0.5) {$\bigtimes$};
  \node at (0.5, 1.5) {$\bigtimes$};
  \node at (1.5, 0.5) {$\bigtimes$};
  \node at (2.5, 1.5) {$\bigtimes$};
  \node at (2.5, 0.5) {$\bigtimes$};
  \draw (0,3.5)--(0,0)--(3.5,0);
  \draw (1,0)--(1,2)--(0,2);
  \draw (0,1)--(1,1);
  \draw (1,0)--(2,0)--(2,1)--(1,1)--(1,0);
  \draw (2,0)--(3,0)--(3,1)--(2,1)--(2,0);
  \draw (2,1)--(3,1)--(3,2)--(2,2)--(2,1);
  \node at (1.75,-0.5) {$D_4$};
\end{tikzpicture}$$
    \caption{The set $KD(D_0)$ for $D_0=\{(1,1),(2,1),(2,2),(2,3),(3,3)\}$}
    \label{fig:KD}
\end{figure}
\end{example}

Finally, as it will be helpful in the sections that follow, we note that one can associate with $\mathrm{KD}(D)$ a natural poset structure. Recall that a \textbf{finite poset} $(\mathcal{P},\preceq_{\mathcal{P}})$ consists of a finite underlying set $\mathcal{P}$ along with a binary relation $\preceq_{\mathcal{P}}$ between the elements of $\mathcal{P}$ which is reflexive, anti-symmetric, and transitive. For a diagram $D$, given $D_1,D_2\in \mathrm{KD}(D)$, we say $D_2\preceq_{\mathrm{KD}(D)} D_1$ if $D_2$ can be obtained from $D_1$ by applying some sequence of Kohnert moves (see~\cite{KP3,KP2}). Ongoing, we denote the poset $(\mathrm{KD}(D),\preceq_{\mathrm{KD}(D)})$ by $\mathcal{P}(D)$, referring to it as the \textbf{Kohnert poset} associated with $D$, and denote $\preceq_{\mathrm{KD}(D)}$ by $\preceq$. Moreover, for $D_1,D_2\in\mathcal{P}(D)$, we write $D_2\prec D_1$ when $D_2\preceq D_1$ and $D_2\neq D_1$, and we write $D_2\precdot D_1$ when $D_2\prec D_1$ and there exists no $D_3\in\mathcal{P}(D)$ for which $D_2\prec D_3\prec D_1$. Intuitively, $D_2 \precdot D_1$ implies that $D_2$ can be obtained from $D_1$ by applying exactly one nontrivial Kohnert move. Note that, by definition, $\mathrm{Min}(D)$ corresponds to the collection of minimal elements in $\mathcal{P}(D)$. A \textbf{chain} in $\mathcal{P}(D)$ is a subset $S\subset \mathcal{P}(D)$ for which given any $D_1,D_2\in S$, either $D_1\preceq D_2$ or $D_2\preceq D_1$, i.e., any two elements of the subset are comparable. For a chain $\Gamma=\{D_0\succ D_1\succ\cdots\succ D_n\}$ of $\mathcal{P}(D)$,
\begin{itemize}
    \item if $\Gamma$ is contained in no larger chain of $\mathcal{P}(D)$, then we refer to $\Gamma$ as a \textbf{maximal chain}; and
    \item if $\Gamma$ is such that for each $j\in [n]$ there exists $r_j\in\mathbb{Z}_{>0}$ satisfying $\mathcal{K}(D_{j-1},r_j)=D_j$, then we refer to $\Gamma$ as a \textbf{Kohnert chain}.
\end{itemize}
Note that if $\Gamma=\{D_0\succ D_1\succ\cdots\succ D_n\}$ is a maximal chain, then it must be the case that $D_{j}\precdot D_{j-1}$ for $j\in [n]$. Thus, since $D_{j}\precdot D_{j-1}$ implies $D_j=\mathcal{K}(D_{j-1},r)$ for some $r$ as noted above, it follows that any maximal chain of $\mathcal{P}(D)$ is also a Kohnert chain. Defining the \textbf{rank} of $\mathcal{P}(D)$, denoted $\mathrm{rk}(D)$, to be one less than the maximal cardinality of a chain in $\mathcal{P}(D)$, we have the following.

\begin{lemma}\label{lem:poset}
    If $D$ is a diagram, then $\mathrm{MC}(D)=\mathrm{rk}(D)$.
\end{lemma}
\begin{proof}
    If $\mathrm{MC}(D)=N$ and $(r_1,\hdots,r_N)\in \mathrm{KSeq}(D,\widehat{D})$ for $\widehat{D}\in\mathrm{Min}(D)$, then $$D=D_0\succ D_1=\mathcal{K}(D_0,r_1)\succ\cdots\succ D_N=\mathcal{K}(D_{N-1},r_N)$$ defines a chain in $\mathcal{P}(D)$ of cardinality $N+1$. Thus, $\mathrm{MC}(D)\le \mathrm{rk}(D)$. On the other hand, evidently, a chain of maximal cardinality must be a maximal chain. As noted above, this implies that the rank of $\mathcal{P}(D)$ is realized by a Kohnert chain. Since a Kohnert chain corresponds to a sequence of nontrivial Kohnert moves which can be applied starting from $D$, the number of which is equal to one less than its cardinality, it follows that $\mathrm{rk}(D)\le \mathrm{MC}(D)$ and the result follows.
\end{proof}


\section{Maximum Number of Moves}\label{sec:max}

In this section, we consider the problem of computing $\MC(D)$. To state the main result of this section, we require the following notation. For a cell $(r, c) \in D$ and $\tilde{c} > 0$, define 
\[\blockers_{D, \tilde{c}}(r, c) = \{(\tilde{r}, \tilde{c}) \in D \mid \tilde{r} \leq r\},\]
i.e., $\blockers_{D,\tilde{c}}(r, c)$ is the collection of cells of $D$ in column $\tilde{c}$ that are weakly below row $r$, and
\[\room_D(r, c) = r - \max_{\tilde{c} \geq c} \{|\blockers_{D, \tilde{c}}(r, c)|\}.\] With this notation, the main result of this section is as follows.

\begin{theorem}\label{thm:max}
    For a diagram $D$,
    \[\MC(D) = \sum_{(r, c) \in D} \room_D(r, c).\]
\end{theorem}

To prove Theorem~\ref{thm:max}, we employ \textbf{labelings}, which are maps $L: D \to \mathbb{Z}_{>0}$. A \textbf{tableau} is a pair $(D, L)$, where $L$ is a labeling of $D$. The \textbf{super-standard} labeling $L_0$ maps $(r, c) \mapsto r$. A labeling $L$ of a diagram $D$ is \textbf{strict} if for all $(r,c),(r',c)\in D$ with $r < r'$, we have $L(r, c) < L(r', c)$. Moreover, $L$ is \textbf{flagged} if $L(r, c) \geq r$ for all $(r, c) \in D$. Observe that the super-standard labeling is strict and flagged by definition. 

Two strict labelings $L, L'$ on diagrams $D, D'$, respectively, are \textbf{column-equivalent} if, for all $c > 0$, the restriction of $L$ to column $c$ has the same image as the restriction of $L'$ to column $c$. In other words, the tableaux $(D, L)$, $(D', L')$ have the same labels in the same columns. Column-equivalence requires that $D$ and $D'$ have identical column weights $\mathrm{cwt}(D)=\mathrm{cwt}(D')$. Given $D \preceq D_0$, the \textbf{standard labeling} of $D$ \textbf{relative to} $D_0$ is the unique strict labeling of $D$ that is column-equivalent to the super-standard labeling $L_0$ on $D_0$. Observe that this labeling is necessarily flagged and that the standard labeling of a diagram relative to itself is super-standard, justifying the usage of the word. Note that not every strict, flagged labeling of $D$ is standard relative to some $D_0$. For an example, see Figure \ref{fig:nonstandard}.

\begin{figure}[ht]
\begin{center}    
    \begin{tikzpicture}[scale=0.5]
    \celx(0,0)[2]
    \celx(1,1)[2]
    \celx(1,0)[1]
    \draw (0,2.5)--(0,0)--(2.5,0);
    \end{tikzpicture}
   \caption{A diagram with a strict, flagged labeling that is not standard relative to any other diagram.}\label{fig:nonstandard} 
  \end{center}
\end{figure}
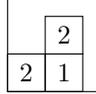

Now, we often wish to focus on a single cell $(r, c)$ of a diagram $D_0$. Therefore, for any $D \preceq D_0$ and fixed $(r, c) \in D_0$, we refer to the unique cell of $D$ in column $c$ with label $r$ (where the labeling is standard relative to $D_0$) as the \textbf{standard $r$-cell} in column $c$ of $D$. In addition, for $D_1\preceq D_0$ and a Kohnert chain $\Gamma_1 = \{D_1 \succ \cdots \succ D_n\}$ in $\mathcal{P}(D_0)$, we set
\[\mathrm{R}_{\Gamma_1}(D_0,r, c) = \{r_1 > r_2 > \cdots > r_k\} = \left\{r_i \in [r] \mid D_{j+1} = D_j\ldownarrow^{(r_i,c)}_{(r',c)} \text{ and } L(r_i, c) = r\right\};\] that is, $\mathrm{R}_{\Gamma_1}(D_0,r, c)$
is the set of rows of Kohnert moves applied which affect the positions of the standard $r$-cells in column $c$ of the diagrams of $\Gamma_1$. Note that for $\Gamma = \{D_0\succ D_1 \succ \cdots \succ D_n\}$, we have
\begin{equation}\label{eq:max}
    |\Gamma|-1 = \sum_{(r, c) \in D_0}|\mathrm{R}_{\Gamma}(D_0,r, c)|.
\end{equation}
In Propositions~\ref{prop:upperbound} and~\ref{prop:upperboundtight} below, for a diagram $D$, a cell $(r,c)\in D$, and a Kohnert chain $\Gamma$ in $\mathcal{P}(D)$ we consider the relationship between $|\mathrm{R}_{\Gamma}(D,r,c)|$ and $\mathrm{room}_D(r,c)$. Specifically, Proposition~\ref{prop:upperbound} shows that $|\mathrm{R}_{\Gamma}(D,r,c)|\le\mathrm{room}_D(r,c)$ in general, while Proposition~\ref{prop:upperboundtight} shows that this upper bound is met for all $(r,c)\in D$ by at least one Kohnert chain in $\mathcal{P}(D)$. Since every maximal chain in $\mathcal{P}(D)$ is a Kohnert chain, considering (\ref{eq:max}) along with Lemma~\ref{lem:poset}, this will suffice to prove Theorem~\ref{thm:max}.

For the proof of Proposition~\ref{prop:upperbound}, we require the following lemma. Within the proof of the lemma and the propositions that follow, for a diagram $D$ and cell $(r,c)\in D$, we say a cell $(\tilde{r}, \tilde{c}) \in D$ with $(\tilde{r}, \tilde{c}) \neq (r, c)$, $\tilde{r} \leq r$, and $\tilde{c} \geq c$ is in the \textbf{zone of attack} of $(r, c)$.

\begin{lemma}\label{lem:roomdecreasing}
    Let $D_0$ be a diagram with $(r, c) \in D_0$, and $D_1,D_2\in\mathrm{KD}(D_0)$ satisfy $D_2\preceq D_1$. If $(r_i, c)$ is the standard $r$-cell in column $c$ of $D_i$ for $i=1$ and $2$ where the labeling is standard relative to $D_0$, then
    \[\room_{D_2}(r_2, c) \leq \room_{D_1}(r_1, c).\]
\end{lemma}

\begin{proof}
    It suffices to show that the result holds when $D_2$ can be obtained from $D_1$ by applying a single Kohnert move, i.e., $$D_2=\mathcal{K}(D_1,s)=D_1\ldownarrow^{(s, d)}_{(s', d)}.$$ If $d < c$, then $(r_1, c) = (r_2, c)$ and the cells have identical zones of attack in their respective diagrams; that is, $r_1=r_2$ and $$\max_{\tilde{c} \geq c} \{|\blockers_{D_2, \tilde{c}}(r_2, c)|\}=\max_{\tilde{c} \geq c} \{|\blockers_{D_1, \tilde{c}}(r_1, c)|\}.$$ Thus, $\room_{D_2}(r_2, c) = \room_{D_1}(r_1, c)$. If $d > c$, then again $(r_1 ,c) = (r_2, c)$, and the number of cells in a given column of the zone of attack of $(r_2,c)$ in $D_2$ is at least the number of cells in the same column of the zone of attack of $(r_1, c)$ in $D_1$; that is, $r_1=r_2$ and $$\max_{\tilde{c} \geq c} \{|\blockers_{D_2, \tilde{c}}(r_2, c)|\}\ge \max_{\tilde{c} \geq c} \{|\blockers_{D_1, \tilde{c}}(r_1, c)|\}.$$ Consequently, $\room_{D_2}(r_2, c) \leq \room_{D_1}(r_1, c)$. 

    Now suppose that $d = c$. If $r_1 = r_2$, then $(r_1, c)$ and $(r_2, c)$ have identical zones of attack in their corresponding diagrams. Thus, as in the case $d<c$, we have $\room_{D_2}(r_2, c) = \room_{D_1}(r_1, c)$. On the other hand, if $r_2 \neq r_1$, then $r_2 = r_1 - 1$. Moreover, $|\blockers_{D_2, c}(r_2, c)| = |\blockers_{D_1, c}(r_1, c)|$ and, for any $\tilde{c} > c$, we have $\blockers_{D_2, \tilde{c}}(r_2, c) = \blockers_{D_1, \tilde{c}}(r_1, c) \setminus \{(r_1, \tilde{c})\}$. Consequently, either $$\max_{\tilde{c} \geq c} \{|\blockers_{D_2, \tilde{c}}(r_2, c)|\}=\max_{\tilde{c} \geq c} \{|\blockers_{D_1, \tilde{c}}(r_1, c)|\}$$ or $$\max_{\tilde{c} \geq c} \{|\blockers_{D_2, \tilde{c}}(r_2, c)|\}=\max_{\tilde{c} \geq c} \{|\blockers_{D_1, \tilde{c}}(r_1, c)|\}-1.$$ In either case, since $r_2=r_1-1$, it follows that $\room_{D_2}(r_2, c)\le \room_{D_1}(r_1, c)$.
\end{proof}

\begin{prop}\label{prop:upperbound}
    If $D_0$ is a diagram, $(r, c) \in D_0$, and $\Gamma = \{D_0\succ D_1 \succ \cdots \succ D_n\}$ is a Kohnert chain of $\mathcal{P}(D_0)$, then $|\mathrm{R}_{\Gamma}(D_0,r, c)| \leq \room_{D_0}(r, c)$.
\end{prop}

\begin{proof}
    Let $(r_j,c)$ be the standard $r$-cell in column $c$ of $D_j$ where the labeling is standard relative to $D_0$, and $\Gamma_j=\{D_j\succ\cdots\succ D_n\}$ for $0\le j\le n$. Note that $\Gamma_0=\Gamma$ and $\Gamma_j$ is a Kohnert chain for $0\le j\le n$. We show that $|\mathrm{R}_{\Gamma_j}(D_0,r, c)| \leq \room_{D_j}(r_j, c)$ for $0\le j\le n$. Our result then follows as the special case $j=0$.

    To start, note that $0=|\mathrm{R}_{\Gamma_n}(D_0,r, c)| \leq \room_{D_n}(r_n, c)$ is immediate as $\room_{D_n}(r_n, c)\ge 0$. Now, assume that $|\mathrm{R}_{\Gamma_j}(D_0,r, c)| \leq \room_{D_j}(r_j, c)$ when $m<j\le n$ for some $m$ satisfying $0\le m<n$. We show that $|\mathrm{R}_{\Gamma_m}(D_0,r, c)| \leq \room_{D_m}(r_m, c)$. If $$D_{m+1}=\mathcal{K}(D_m,\widehat{r})=D_{m}\ldownarrow^{(\widehat{r},c')}_{(r',c')}$$ and $(\widehat{r},c')\neq (r_{m},c)$, then $|R_{\Gamma_{m}}(D_0,r,c)|=|R_{\Gamma_{m+1}}(D_0,r,c)|\le \room_{D_{m+1}}(r_{m+1}, c)$. Applying Lemma~\ref{lem:roomdecreasing}, it follows that $|R_{\Gamma_{m}}(D_0,r,c)|\le \room_{D_{m+1}}(r_{m+1}, c)\le \room_{D_{m}}(r_{m}, c)$, as desired. So, assume that $$D_{m+1}=\mathcal{K}(D_m,r_m)=D_{m}\ldownarrow^{(r_m,c)}_{(r',c)}.$$ Observe that $|R_{\Gamma_{m}}(D_0,r,c)|=|R_{\Gamma_{m+1}}(D_0,r,c)|+1$ and the standard $r$-cell in column $c$ of $D_{m+1}$ is in position $(r_{m+1},c)=(r_m - 1, c)$. Additionally, since $(r_m,c)$ is the rightmost cell in row $r_m$ of $D_m$, any $(r'', c') \in D_m$ in the zone of attack of $(r_m,c)$ has $r'' < r_m$. Thus, 
    \[\room_{D_{m+1}}(r_{m+1}, c)=\room_{D_{m+1}}(r_m - 1, c) = \room_{D_m}(r_m, c) - 1.\] Consequently, \[|\mathrm{R}_{\Gamma_m}(D_0,r, c)| = |\mathrm{R}_{\Gamma_{m+1}}(D_0,r, c)| + 1 \leq \room_{D_{m+1}}(r_{m+1},c) + 1 = \room_{D_m}(r_m, c),\]
    as desired. Therefore, $|\mathrm{R}_{\Gamma_m}(D_0,r, c)| \leq \room_{D_m}(r_m, c)$ and we may conclude that $|\mathrm{R}_{\Gamma_j}(D_0,r, c)| \leq \room_{D_j}(r_j, c)$ for $0\le j\le n$. As noted above, the result follows.
\end{proof}

Let $r_D$ denote the minimal row of a diagram $D$ to which a nontrivial Kohnert move can be applied, or 0 otherwise. Then, starting with a diagram $D_0$, we recursively define
\[D_{i+1} = \begin{cases}
    \mathcal{K}(D_i, r_{D_i}), & r_{D_i} > 0\\
    D_i, & \text{else.}
\end{cases}\]
Observe that, eventually, $D_{i+1} = D_i$. Let $N$ denote the smallest such index. Then $D_N \in \mathrm{Min}(D_0)$ and we say $\Gamma(D_0) = \{D_0 \succ D_1 \succ \cdots \succ D_N\}$ is the Kohnert chain generated by this operation.

\begin{prop}\label{prop:upperboundtight}
    Given a diagram $D_0$ and a cell $(r, c) \in D_0$, we have $|\mathrm{R}_{\Gamma(D_0)}(D_0,r, c)| = \room_{D_0}(r, c)$.
\end{prop}

\begin{proof}
Let $\Gamma(D_0) = \{D_0 \succ D_1 \succ \cdots \succ D_N\}$, $(r_j,c)$ be the standard $r$-cell in column $c$ of $D_j$ where the labeling is standard relative to $D_0$, and $\Gamma_j=\{D_j\succ\cdots\succ D_n\}$ for $0\le j\le N$. Note that $\Gamma_0=\Gamma(D_0)$. We show that $|\mathrm{R}_{\Gamma_j}(D_0,r, c)| = \room_{D_j}(r_j, c)$ for $0\le j\le N$. Our result then follows as the special case $j=0$.

To start, note that $0=|R_{\Gamma_N}(D_0,r,c)|$. Moreover, by the definition of $\Gamma(D_0)$, we must have that $\mathcal{K}(D_N,r_N)=D_N$. Consequently, since $(r_N,c)\in D_N$, there exists $\tilde{c}\ge c$ such that $(\tilde{c},\tilde{r})\in D_N$ for all $1\le \tilde{r}\le r_N$, i.e., $\room_{D_N}(r_N, c)=0=|R_{\Gamma_N}(D_0,r,c)|$, as desired. Now, assume that $|\mathrm{R}_{\Gamma_j}(D_0,r, c)| = \room_{D_j}(r_j, c)$ when $m<j\le N$ for some $m$ satisfying $0\le m<N$. We show that $|\mathrm{R}_{\Gamma_m}(D_0,r, c)| = \room_{D_m}(r_m, c)$. There are two cases.
\bigskip

\noindent
\textbf{Case 1:} $D_{m+1}=\mathcal{K}(D_m,\widehat{r})$ with $\widehat{r}\neq r_m$. Note that in this case we have $|\mathrm{R}_{\Gamma_m}(D_0,r, c)|=|\mathrm{R}_{\Gamma_{m+1}}(D_0,r, c)|$. Moreover, we claim that $\room_{D_m}(r_m, c)=\room_{D_{m+1}}(r_m+1, c)$ as well, from which the result follows since $|\mathrm{R}_{\Gamma_{m+1}}(D_0,r, c)|=\room_{D_{m+1}}(r_m+1, c)$ by assumption. To see this, first note that if $r_m<\widehat{r}$, then $\room_{D_m}(r_m, c)=\room_{D_{m+1}}(r_m+1, c)$ follows immediately since $r_m=r_{m+1}$ and the number of cells in the zones of attack of $(r_m,c)=(r_{m+1},c)$ in $D_m$ and $D_{m+1}$ remain the same. On the other hand, if $\widehat{r}>r_m$, then, by the definition of $\Gamma(D_0)$, we must have $D_m=\mathcal{K}(D_m,r_m)$. Thus, since $(r_m,c)\in D_m$, it follows that there exists $\tilde{c}\ge c$ for which $(\tilde{r},\tilde{c})\in D_m$ for all $1\le \tilde{r}\le r_m$. Consequently, $\room_{D_m}(r_m, c)=0$. Applying Lemma~\ref{lem:roomdecreasing}, it then follows that $\room_{D_{m+1}}(r_{m+1}, c)\le \room_{D_m}(r_m, c)=0$, i.e., $\room_{D_{m+1}}(r_{m+1}, c)=0=\room_{D_m}(r_m, c)$, as desired.
\bigskip

\noindent
\textbf{Case 2:} $D_{m+1}=\mathcal{K}(D_m,r_m)$. Note that in this case, since $(r_m,c)\in D_m$, we have $$D_{m+1}=D_m\ldownarrow^{(r_m,c')}_{(r',c')}$$ with $c'\ge c$. If $c'=c$, then $|\mathrm{R}_{\Gamma_m}(D_0,r, c)|=|\mathrm{R}_{\Gamma_{m+1}}(D_0,r, c)|+1$ and $r_{m+1}=r_m-1$. Additionally, since $(r_m,c)$ is the rightmost cell in row $r_m$ of $D_m$, any $(r'', c') \in D_m$ in the zone of attack of $(r_m,c)$ has $r'' < r_m$. Consequently, $$\room_{D_m}(r_m, c)=\room_{D_{m+1}}(r_{m+1}, c)+1=|\mathrm{R}_{\Gamma_{m+1}}(D_0,r, c)|+1=|\mathrm{R}_{\Gamma_m}(D_0,r, c)|,$$ where we have used the assumption that $\room_{D_{m+1}}(r_{m+1}, c)=|\mathrm{R}_{\Gamma_{m+1}}(D_0,r, c)|$. On the other hand, if $c'>c$, then $|\mathrm{R}_{\Gamma_m}(D_0,r, c)|=|\mathrm{R}_{\Gamma_{m+1}}(D_0,r, c)|$, $r_m=r_{m+1}$, and the number of cells in the zones of attack of $(r_m,c)=(r_{m+1},c)$ in $D_m$ and $D_{m+1}$ are the same, i.e., $\room_{D_m}(r_m, c)=\room_{D_{m+1}}(r_{m+1}, c)$. Consequently, by assumption, $$\room_{D_m}(r_m, c)=\room_{D_{m+1}}(r_{m+1}, c)=|\mathrm{R}_{\Gamma_{m+1}}(D_0,r, c)|=|\mathrm{R}_{\Gamma_m}(D_0,r, c)|,$$ as desired.
\bigskip

\noindent
Therefore, $|\mathrm{R}_{\Gamma_m}(D_0,r, c)| = \room_{D_m}(r_m, c)$ and we may conclude that $|\mathrm{R}_{\Gamma_j}(D_0,r, c)| = \room_{D_j}(r_j, c)$ for $0\le j\le N$. As noted above, the result follows.
\end{proof}

We are now ready to prove Theorem \ref{thm:max}.
\begin{proof}[Proof of Theorem \ref{thm:max}]
    As noted in (\ref{eq:max}), for any maximal chain $\Gamma \subseteq \mathcal{P}(D_0)$, 
    \[|\Gamma|-1 = \sum_{(r, c) \in D_0}|\mathrm{R}_{\Gamma}(D_0,r, c)|.\]
    Thus, by Proposition~\ref{prop:upperbound},
    \[|\Gamma|-1 \leq \sum_{(r, c) \in D_0} \room_{D_0}(r, c).\]
    However, Proposition~\ref{prop:upperboundtight} has
    \[|\Gamma(D_0)|-1 = \sum_{(r, c) \in D_0} \room_{D_0}(r, c).\]
    Consequently, $\Gamma(D_0)$ is a longest chain in $\mathcal{P}(D_0)$ and it has the desired length. Applying Lemma~\ref{lem:poset}, the result follows.
\end{proof}

\begin{example}
    Let $D$ be the diagram illustrated in Figure~\ref{fig:MCex} $(a)$. To compute $\mathrm{MC}(D)$, in Figure~\ref{fig:MCex} $(b)$ we include row labels to the right of each row and label each cell $(r,c)\in D$ by $\max_{\tilde{c} \geq c} \{|\blockers_{D, \tilde{c}}(r, c)|\}$. Then, by Theorem~\ref{thm:max}, for each labeled cell in Figure~\ref{fig:MCex} $(b)$, we subtract the label of each cell from the row that the cell occupies and sum up the resulting values. Here, we get 
    $$\mathrm{MC}(D)=3(2-1)+(7-4)+3(5-3)+2(4-2)+3(3-1)+(7-3)+3(6-2)+(4-1)=41.$$
    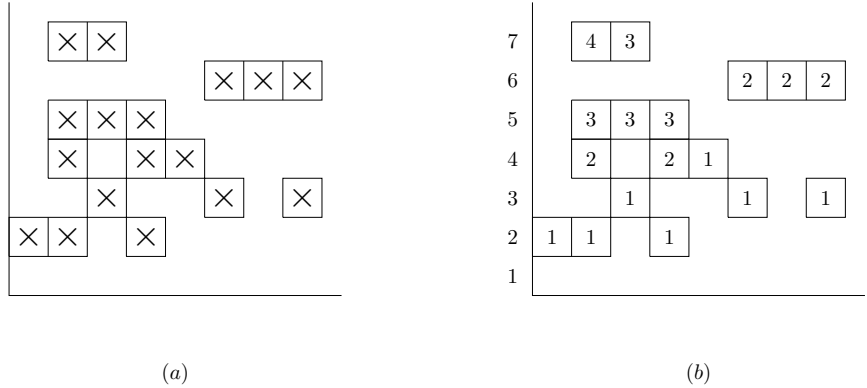
\begin{figure}[H]
        \centering
        $$\scalebox{0.8}{\begin{tikzpicture}[scale=0.65]
    \node at (0.5, 1.5) {$\bigtimes$};
        \node at (1.5, 1.5) {$\bigtimes$};
        \node at (2.5, 2.5) {$\bigtimes$};
        \node at (5.5, 2.5) {$\bigtimes$};
        \node at (1.5, 3.5) {$\bigtimes$};
        \node at (3.5, 3.5) {$\bigtimes$};
        \node at (1.5, 4.5) {$\bigtimes$};
        \node at (2.5, 4.5) {$\bigtimes$};
        \node at (1.5, 6.5) {$\bigtimes$};
        \node at (5.5, 5.5) {$\bigtimes$};
        \node at (6.5, 5.5) {$\bigtimes$};

        \node at (7.5, 2.5) {$\bigtimes$};
        \node at (7.5, 5.5) {$\bigtimes$};

        \node at (4.5, 3.5) {$\bigtimes$};

        \node at (3.5, 4.5) {$\bigtimes$};

        \node at (3.5, 1.5) {$\bigtimes$};

        \node at (2.5, 6.5) {$\bigtimes$};

        \draw (0,7.5)--(0,0)--(8.5,0);
        
        \draw (0,1)--(2,1)--(2,2)--(0,2)--(0,1);
        \draw (1,1)--(1,2);
        \draw (3,1)--(4,1)--(4,2)--(3,2)--(3,1);
        
        \draw (2,2)--(3,2)--(3,3)--(2,3)--(2,2);
        \draw (5,2)--(6,2)--(6,3)--(5,3)--(5,2);
        \draw (7,2)--(8,2)--(8,3)--(7,3)--(7,2);

        \draw (1,3)--(2,3)--(2,4)--(1,4)--(1,3);

        \draw (3,3)--(5,3)--(5,4)--(3,4)--(3,3);
        \draw (4,3)--(4,4);

        \draw (1,4)--(4,4)--(4,5)--(1,5);
        \draw (1,4)--(1,5);
        \draw (2,4)--(2,5);
        \draw (3,4)--(3,5);
        \draw (5,5)--(8,5)--(8,6)--(5,6)--(5,5);
        \draw (6,5)--(6,6);
        \draw (7,5)--(7,6);
        \draw (1,6)--(3,6)--(3,7)--(1,7)--(1,6);
        \draw (2,6)--(2,7);
        \node at (4.25, -2) {$(a)$};
\end{tikzpicture}}\quad\quad\quad\quad\quad\quad \scalebox{0.8}{\begin{tikzpicture}[scale=0.65]
    \node at (0.5, 1.5) {$1$};
        \node at (1.5, 1.5) {$1$};
        \node at (2.5, 2.5) {$1$};
        \node at (5.5, 2.5) {$1$};
        \node at (1.5, 3.5) {$2$};
        \node at (3.5, 3.5) {$2$};
        \node at (1.5, 4.5) {$3$};
        \node at (2.5, 4.5) {$3$};
        \node at (1.5, 6.5) {$4$};
        \node at (5.5, 5.5) {$2$};
        \node at (6.5, 5.5) {$2$};
        
        \node at (-0.5, 0.5) {$1$};
        \node at (-0.5, 1.5) {$2$};
        \node at (-0.5, 2.5) {$3$};
        \node at (-0.5, 3.5) {$4$};
        \node at (-0.5, 4.5) {$5$};
        \node at (-0.5, 5.5) {$6$};
        \node at (-0.5, 6.5) {$7$};

        \node at (7.5, 2.5) {$1$};
        \node at (7.5, 5.5) {$2$};

        \node at (4.5, 3.5) {$1$};

        \node at (3.5, 4.5) {$3$};

        \node at (3.5, 1.5) {$1$};

        \node at (2.5, 6.5) {$3$};

        \draw (0,7.5)--(0,0)--(8.5,0);
        
        \draw (0,1)--(2,1)--(2,2)--(0,2)--(0,1);
        \draw (1,1)--(1,2);
        \draw (3,1)--(4,1)--(4,2)--(3,2)--(3,1);
        
        \draw (2,2)--(3,2)--(3,3)--(2,3)--(2,2);
        \draw (5,2)--(6,2)--(6,3)--(5,3)--(5,2);
        \draw (7,2)--(8,2)--(8,3)--(7,3)--(7,2);

        \draw (1,3)--(2,3)--(2,4)--(1,4)--(1,3);

        \draw (3,3)--(5,3)--(5,4)--(3,4)--(3,3);
        \draw (4,3)--(4,4);

        \draw (1,4)--(4,4)--(4,5)--(1,5);
        \draw (1,4)--(1,5);
        \draw (2,4)--(2,5);
        \draw (3,4)--(3,5);
        \draw (5,5)--(8,5)--(8,6)--(5,6)--(5,5);
        \draw (6,5)--(6,6);
        \draw (7,5)--(7,6);
        \draw (1,6)--(3,6)--(3,7)--(1,7)--(1,6);
        \draw (2,6)--(2,7);
        
        \node at (4.25, -2) {$(b)$};
\end{tikzpicture}}$$
        \caption{Example computing $\mathrm{MC}(D)$}
        \label{fig:MCex}
    \end{figure}
\end{example}

Considering the work above, we are led to the following algorithm for solving the maximum Kohnert move puzzle.

\begin{tcolorbox}[breakable, enhanced]
\centerline{\textbf{\underline{Solution to Maximum Kohnert Move Puzzle}}}
\bigskip

Given a diagram $D_1=D$ proceed as follows:
\begin{itemize}
    \item[] Step $i$: If $\mathcal{K}(D_i,r)=D_0$ for all $r>0$, then $D_i$ is minimal and the puzzle has been solved. Otherwise, letting $r_{D_i}$ be minimal such that  $\mathcal{K}(D_i,r_{D_i})\neq D_i$, set $D_{i+1}=\mathcal{K}(D_i,r_{D_i})$ and go to Step $i+1$.
\end{itemize}
Note that to solve the puzzle we apply Kohnert moves to lower the bottom-most cell whose position can be affected by a Kohnert move at each step.
\end{tcolorbox}

\section{Minimum Number of Moves}\label{sec:min}

In this section, we consider the problem of computing $\mathrm{mc}(D)$. To state the main result of this section, we require the following notation. Suppose that $D$ is a diagram with $\mathrm{cwt}(D)=(m_1,\hdots,m_n)$. For each nonempty column $c$ of $D$, letting 
\begin{itemize}
    \item $M=\max\{m_i~|~c<i\le n\}$, i.e., $M$ is the maximum number of cells contained in a column lying strictly to the right of column $c$ in $D$, and
    \item $r$ denote the row occupied by the topmost cell in column $c$ of $D$,
\end{itemize}
define $$h(D,c)=\begin{cases}
    r, & M>r \\
    m_c, & m_c\ge M \\
    M, & r>M> m_c.
\end{cases}.$$ We will show in Proposition~\ref{prop:highestcellmin}, that $$h(D,c)=\max\{r~|~(r,c)\in\tilde{D}\in\mathrm{Min}(D)\},$$ i.e., $h(D,c)$ corresponds to the highest possible row occupied by the topmost cell in column $c$ of a diagram belonging to $\textrm{Min}(D)$. With the notation above, we now state the main result of this section.

\begin{theorem}\label{thm:minmain}
    Let $D$ be a diagram with nonempty columns $c_1<c_2<\cdots<c_n$. For $i\in [n]$, suppose that $(r_i,c_i)$ is the topmost cell in column $c_i$ of $D$. Then $$\mathrm{mc}(D)=\sum_{i=1}^n[r_i-h(D,c_i)].$$
\end{theorem}

To establish Theorem~\ref{thm:minmain}, we first show that $\mathrm{mc}(D)$ is the difference of the number of empty positions in $D$ and the maximum number of empty positions in a diagram of $\mathrm{Min}(D)$. In particular, for a diagram $D$, we show that 
\begin{equation*}
    \mathrm{mc}(D)=|\mathrm{empty}(D)|-\max_{\widehat{D}\in\mathrm{Min}(D)}|\mathrm{empty}(\widehat{D})|.
\end{equation*}
in Theorem~\ref{thm:min}. To do so, we start by forming a better understanding of the value $$\max_{\widehat{D}\in\mathrm{Min}(D)}|\mathrm{empty}(\widehat{D})|.$$ Note that for a diagram $D$ with nonempty column $c$, the number of empty positions in column $c$ of $D$ is given by the row occupied by the topmost cell in column $c$ of $D$ minus the total number of cells in column $c$ of $D$. Thus, to understand $\max_{\widehat{D}\in\mathrm{Min}(D)}|\mathrm{empty}(\widehat{D})|$, it suffices to understand the maximum rows occupied by cells in nonempty columns of minimal diagrams associated with $D$.

\begin{lemma}\label{lem:minprophelp}
    Let $D$ be a diagram with $\mathrm{cwt}(D)=(m_1,\hdots,m_n)$. Suppose that $M_j=\max\{m_i~|~j\le i\le n\}$ for $1\le j\le n$. Then for all $\tilde{D}\in\mathrm{Min}(D)$ and $1\le j\le n$,
    \begin{enumerate}
        \item[$(i)$] the cells in column $c$ of $\tilde{D}$ for $j\le c\le n$ occupy rows weakly below row $M_j$ and
        \item[$(ii)$] for each $1 \le \tilde{r} \le M_j$, there exists at least one column $c^*$ for which $j\le c^*\le n$ and $(\tilde{r},c^*)\in \tilde{D}$.
    \end{enumerate}
\end{lemma}
\begin{proof}
    For $(i)$, assume otherwise; that is, assume that there exists $D^*\in\mathrm{Min}(D)$ and $j\le c^*\le n$ for which the topmost cell in column $c^*$ of $D^*$ occupies a row $r^*>M_j$. Assume that column $c^*$ is the rightmost such column. Then $(r^*,c^*)\in D^*$ is the rightmost cell in row $r^*$ of $D^*$. Moreover, since the number of cells in column $c^*$ of $D^*$ must be strictly less than $r^*$ by assumption, there exists an empty position strictly below $(r^*,c^*)$ in column $c^*$ of $D^*$. Thus, $\mathcal{K}(D^*,r^*)\neq D^*$, contradicting our assumption that $D^*\in\mathrm{Min}(D)$.

    Now, for $(ii)$, assume that $c$ for $j\le c\le n$ is a column of $D$ containing $M_j$ cells. Then, considering $(i)$, for every $\tilde{D}\in \mathrm{Min}(D)$ and $1 \le \tilde{r} \le M_j$, we have $(\tilde{r},c)\in \tilde{D}$.
\end{proof}

Let $D$ be a diagram with nonempty columns $c_1<c_2<\hdots<c_n$ where $(r_i,c_i)\in D$ is the topmost cell in column $c_i$ of $D$ for $i\in[n]$. Define the diagrams $\widehat{D}_{c_i}\in \mathrm{KD}(D)$ for $i\in [n]$ iteratively as follows: Setting $\widehat{D}_{c_{n+1}}=D$, for $i\in [n]$, we define $\widehat{D}_{c_i}$ to be the diagram formed from $\widehat{D}_{c_{i+1}}$ by successively applying one Kohnert move at rows $r_i$ down to 1 in decreasing order. Intuitively, $\widehat{D}_{c_i}$ for $i\in [n]$ is formed by attempting to bottom justify the cells in columns $\tilde{c}>c_i$ as best as possible working from right to left. Note that, in general, many of the Kohnert moves applied in forming $\widehat{D}_{c_i}$ may be trivial. For the sake of brevity, such moves are not omitted. For an illustration of this construction, see Example~\ref{ex:lower}.

\begin{example}\label{ex:lower}
In Figure~\ref{fig:lower} below, for $D$ the leftmost diagram, we illustrate $\widehat{D}_4$, $\widehat{D}_3$, $\widehat{D}_2$, and $\widehat{D}_1$. Note that $\widehat{D}_4$ is formed from $D$ by successively applying a single Kohnert moves at rows $4$ down to $1$ in decreasing order; $\widehat{D}_3$ is formed from $\widehat{D}_4$ by successively applying a single Kohnert moves at rows $3$ down to $1$ in decreasing order; $\widehat{D}_2$ is formed from $\widehat{D}_3$ by successively applying a single Kohnert moves at rows $2$ down to $1$ in decreasing order \textup(each such move being trivial\textup); and $\widehat{D}_1$ is formed from $\widehat{D}_2$ by successively applying a single Kohnert moves at rows $3$ down to $1$ in decreasing order.
\begin{figure}[H]
    \centering
    $$\begin{tikzpicture}[scale=0.65]
    \draw (0,4.5)--(0,0)--(4.5,0);
    \node at (0.5, 2.5) {$\bigtimes$};
  \node at (0.5, 3.5) {$\bigtimes$};
    \draw (0,2)--(1,2)--(1,3)--(0,3);
    \draw (0,3)--(1,3)--(1,4)--(0,4);
    \node at (1.5, 1.5) {$\bigtimes$};
    \draw (1,1)--(2,1)--(2,2)--(1,2)--(1,1);
    \node at (2.5, 1.5) {$\bigtimes$};
    \node at (2.5, 2.5) {$\bigtimes$};
    \node at (2.5, 3.5) {$\bigtimes$};
    \draw (2,2)--(3,2)--(3,3)--(2,3)--(2,2);
    \draw (2,3)--(3,3)--(3,4)--(2,4)--(2,3);
    \draw (2,1)--(3,1)--(3,2)--(2,2)--(2,1);
    \node at (3.5, 3.5) {$\bigtimes$};
    \node at (3.5, 1.5) {$\bigtimes$};
    \draw (3,3)--(4,3)--(4,4)--(3,4)--(3,3);
    \draw (3,1)--(4,1)--(4,2)--(3,2)--(3,1);
  \node at (2,-1) {$D$};
\end{tikzpicture}\quad\quad\quad\quad \begin{tikzpicture}[scale=0.65]
    \draw (0,4.5)--(0,0)--(4.5,0);
    \node at (0.5, 2.5) {$\bigtimes$};
  \node at (0.5, 3.5) {$\bigtimes$};
    \draw (0,2)--(1,2)--(1,3)--(0,3);
    \draw (0,3)--(1,3)--(1,4)--(0,4);
    \node at (1.5, 1.5) {$\bigtimes$};
    \draw (1,1)--(2,1)--(2,2)--(1,2)--(1,1);
    \node at (2.5, 1.5) {$\bigtimes$};
    \node at (2.5, 2.5) {$\bigtimes$};
    \node at (2.5, 3.5) {$\bigtimes$};
    \draw (2,2)--(3,2)--(3,3)--(2,3)--(2,2);
    \draw (2,3)--(3,3)--(3,4)--(2,4)--(2,3);
    \draw (2,1)--(3,1)--(3,2)--(2,2)--(2,1);
    \node at (3.5, 1.5) {$\bigtimes$};
    \node at (3.5, 0.5) {$\bigtimes$};
    \draw (3,1)--(4,1)--(4,2)--(3,2)--(3,1);
    \draw (3,0)--(4,0)--(4,1)--(3,1)--(3,0);
  \node at (2,-1) {$\widehat{D}_4$};
\end{tikzpicture}\quad\quad\quad\quad \begin{tikzpicture}[scale=0.65]
    \draw (0,4.5)--(0,0)--(4.5,0);
    \node at (0.5, 2.5) {$\bigtimes$};
  \node at (0.5, 3.5) {$\bigtimes$};
    \draw (0,2)--(1,2)--(1,3)--(0,3);
    \draw (0,3)--(1,3)--(1,4)--(0,4);
    \node at (1.5, 1.5) {$\bigtimes$};
    \draw (1,1)--(2,1)--(2,2)--(1,2)--(1,1);
    \node at (2.5, 0.5) {$\bigtimes$};
    \node at (2.5, 1.5) {$\bigtimes$};
    \node at (2.5, 2.5) {$\bigtimes$};
    \draw (2,1)--(3,1)--(3,2)--(2,2)--(2,1);
    \draw (2,2)--(3,2)--(3,3)--(2,3)--(2,2);
    \draw (2,0)--(3,0)--(3,1)--(2,1)--(2,0);
    \node at (3.5, 1.5) {$\bigtimes$};
    \node at (3.5, 0.5) {$\bigtimes$};
    \draw (3,1)--(4,1)--(4,2)--(3,2)--(3,1);
    \draw (3,0)--(4,0)--(4,1)--(3,1)--(3,0);
  \node at (2,-1) {$\widehat{D}_3=\widehat{D}_2$};
\end{tikzpicture}\quad\quad\quad\quad \begin{tikzpicture}[scale=0.65]
    \draw (0,4.5)--(0,0)--(4.5,0);
    \node at (0.5, 2.5) {$\bigtimes$};
  \node at (0.5, 1.5) {$\bigtimes$};
    \draw (0,2)--(1,2)--(1,3)--(0,3);
    \draw (0,1)--(1,1)--(1,2)--(0,2);
    \node at (1.5, 1.5) {$\bigtimes$};
    \draw (1,1)--(2,1)--(2,2)--(1,2)--(1,1);
    \node at (2.5, 0.5) {$\bigtimes$};
    \node at (2.5, 1.5) {$\bigtimes$};
    \node at (2.5, 2.5) {$\bigtimes$};
    \draw (2,1)--(3,1)--(3,2)--(2,2)--(2,1);
    \draw (2,2)--(3,2)--(3,3)--(2,3)--(2,2);
    \draw (2,0)--(3,0)--(3,1)--(2,1)--(2,0);
    \node at (3.5, 1.5) {$\bigtimes$};
    \node at (3.5, 0.5) {$\bigtimes$};
    \draw (3,1)--(4,1)--(4,2)--(3,2)--(3,1);
    \draw (3,0)--(4,0)--(4,1)--(3,1)--(3,0);
  \node at (2,-1) {$\widehat{D}_1$};
\end{tikzpicture}$$
    \caption{$D$ and $\widehat{D}_i$ for $i\in [4]$}
    \label{fig:lower}
\end{figure}
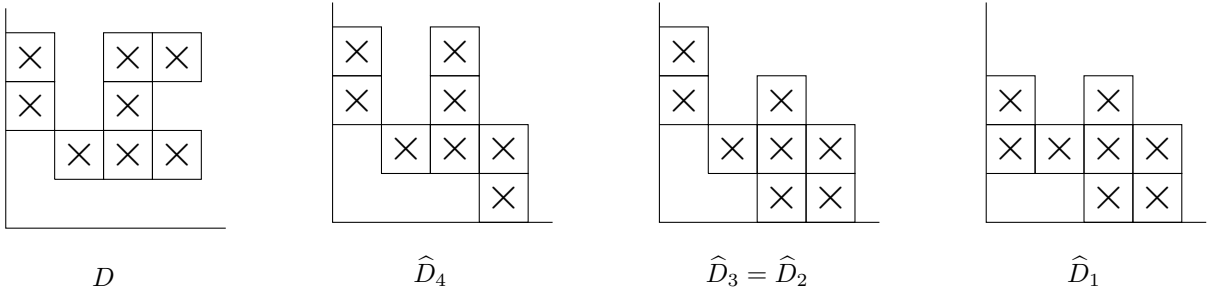
\end{example}



\begin{lemma}\label{lem:minprophelp2}
Let $D$ be a diagram and $c$ be a nonempty column of $D$. Then
\begin{enumerate}
    \item[$(i)$] for all $\tilde{c} < c$, $(r, \tilde{c}) \in \widehat{D}_c$ if and only if $(r, \tilde{c}) \in D$, i.e., $\widehat{D}_c$ and $D$ match in columns lying strictly to the left of column $c$;
    \item[$(ii)$] for all $\tilde{D}\preceq\widehat{D}_c$ and $\tilde{c}\ge c$, $(r,\tilde{c})\in \tilde{D}$ if and only if $(r,\tilde{c})\in\widehat{D}_c$, i.e., $\widehat{D}_c$ and $\tilde{D}$ match in columns lying weakly to the right of column $c$; and
    \item[$(iii)$] the topmost cell in column $c$ of $\widehat{D}_c$ occupies row $h(D,c)$.
\end{enumerate}
\end{lemma}
\begin{proof}
    Assume that $\mathrm{cwt}(D)=(m_1,\hdots,m_n)$. Moreover, for $i\in [n]$ with $m_i>0$, suppose that $(r_i,i)\in D$ is the topmost cell in column $i$ of $D$. Starting with $\widehat{D}_{n}$, note that by definition, if $r_n>m_n$, then 
    \begin{itemize}
        \item when applying Kohnert moves at rows $r_n$ down to $m_n+1$ in decreasing order the topmost cell in column $n$ of the corresponding diagram moves to a lower row; and
        \item after applying the Kohnert move at row $m_n+1$, the cells in column $n$ are bottom justified so that the remaining Kohnert moves applied in forming $\widehat{D}_{n}$ do nothing.
    \end{itemize}
    In the case that $r_n=m_n$, then the cells in column $n$ are bottom justified in $D$, $D=\widehat{D}_{n}$ and all Kohnert moves applied in forming $\widehat{D}_{n}$ do nothing. Consequently, in either case, $(i)$ through $(iii)$ hold.

    Now, assume the result holds for $\widehat{D}_{j}$ for $1<j\le n$. If $m_{j-1}=0$, then $\widehat{D}_{j-1}=\widehat{D}_{j}$ and, evidently, $(i)$ through $(iii)$ hold for $\widehat{D}_{j-1}$. So, assume that $m_{j-1}>0$. Let $M$ denote $\max\{m_i~|~j\le i\le n\}$. Note that since properties $(i)$-$(iii)$ hold in $\widehat{D}_{j}$, $$(\ast)\quad \text{the rightmost cells in rows $1\le r\le M$ of $\widehat{D}_j$ must contain no empty positions below}.$$ To see this, note that otherwise there would exist $(r,k)\in \widehat{D}_{j}$ with $j\le k\le n$ which is rightmost in its row and for which there is a max $r^*<r$ such that $(r^*,k)$ is an empty position in $\widehat{D}_{j}$; but then $$\widehat{D}_{j}\succeq D^*=\mathcal{K}(\widehat{D}_{j},r)=\widehat{D}_{j}\ldownarrow^{(r,k)}_{(r^*,k)},$$ contradicting our assumption that $\widehat{D}_{j}$ satisfies property $(ii)$. Now, recall that $\widehat{D}_{j-1}$ is formed from $\widehat{D}_{j}$ by sequentially applying Kohnert moves at rows $r_{j-1}$ down to $1$ in decreasing order. There are three cases.
    \bigskip

    \noindent
    \textbf{Case 1:} $M\ge r_{j-1}$. In this case, considering $(\ast)$, it follows that $\widehat{D}_{j}=\widehat{D}_{j-1}$ so that property $(i)$ holds for $\widehat{D}_{j-1}$ since it holds for $\widehat{D}_{j}$. Moreover, the topmost cell in column $j-1$ of $\widehat{D}_{j-1}$ occupies row $r_{j-1}=h(\widehat{D}_{j-1},j-1)$, so that property $(iii)$ holds for $\widehat{D}_{j-1}$. Finally, considering Lemma~\ref{lem:minprophelp} (i), since property $(ii)$ holds in $\widehat{D}_j$, it follows that all cells in columns $c$ of $\widehat{D}_j$ satisfying $j\le c\le n$ lie weakly below row $M$. Thus, since $r_{j-1}\le M$ and $\widehat{D}_{j}=\widehat{D}_{j-1}$, once again considering $(\ast)$, it follows that property $(ii)$ holds in $\widehat{D}_{j-1}$ as well.
    \bigskip

    \noindent
    \textbf{Case 2:} $r_{j-1}>M$. Assume that $r_{j-1}>m_{j-1}$; the other case following via a similar but simpler argument. In this case, while forming $\widehat{D}_{j-1}$ from $\widehat{D}_{j}$, when applying Kohnert moves at rows $r_{j-1}$ down to $m_{j-1}+1$ in decreasing order, the topmost cell in column $j-1$ of the corresponding diagram moves to a lower row. After applying the Kohnert move at row $m_{j-1}+1$, the cells of column $j-1$ are bottom justified and, considering $(\ast)$, the remaining Kohnert moves applied in forming $\widehat{D}_{j-1}$ do nothing. Consequently, property $(i)$ holds. Moreover, the topmost cell in column $j-1$ of $\widehat{D}_{j-1}$ occupies row $m_{j-1}=h(D,j-1)$ so that property $(iii)$ holds in $\widehat{D}_{j-1}$. Finally, note that the rightmost cell in rows $1$ up to $m_{j-1}>M$ in $\widehat{D}_{j-1}$ belong to either column $j-1$ which is bottom justified, or column $k$ for $j-1<k\le n$. Thus, once again considering $(\ast)$, $\mathcal{K}(\widehat{D}_{j-1},\tilde{r})=\widehat{D}_{j-1}$ for $1\le \tilde{r}\le m_{j-1}$. As, arguing in a similar manner to as in Case 1, all cells in columns $c$ for $j-1\le c\le n$ lie weakly below row $m_{j-1}$, it follows that $(ii)$ holds in $\widehat{D}_{j-1}$.
    \bigskip

    \noindent
    \textbf{Case 3:} $r_{j-1}>M>m_{j-1}$. In this case, while forming $\widehat{D}_{j-1}$ from $\widehat{D}_{j}$, when applying Kohnert moves at rows $r_{j-1}$ down to $M+1$ in decreasing order, the topmost cell in column $j-1$ of the corresponding diagram moves to a lower row. Since $M>m_{j-1}$, after applying the Kohnert move at row $M+1$, the topmost cell of column $j-1$ occupies row $M=h(D,j-1)$. Considering $(\ast)$, the remaining Kohnert moves applied in forming $\widehat{D}_{j-1}$ from $\widehat{D}_{j}$ do nothing. Consequently, we have that both properties $(i)$ and $(iii)$ hold in $\widehat{D}_{j-1}$. Now, arguing in a similar manner to as in Case 1, we have that all cells in columns $c$ for $j-1\le c\le n$ lie weakly below row $M$. Thus, since the rightmost cell in rows $1\le r\le M$ belong to a column $k$ for $j-1<k\le n$, once again considering $(\ast)$, it follows that property $(iii)$ holds in $\widehat{D}_{j-1}$.
    \bigskip

    \noindent
    The result follows.
\end{proof}

Using the results above, we now prove Proposition~\ref{prop:highestcellmin} which, as noted above, shows that $h(D,c)$ corresponds to the highest possible row occupied by the topmost cell in column $c$ of a diagram belonging to $\textrm{Min}(D)$.

\begin{prop}\label{prop:highestcellmin}
    Let $D$ be a diagram and suppose that column $c$ of $D$ is nonempty. Then $$h(D,c)=\max\{r~|~(r,c)\in \tilde{D}\in\mathrm{Min}(D)\}.$$
\end{prop}
\begin{proof}
    Observe that if column $c$ is the rightmost nonempty column of $D$, then the result is immediate. So, assume that column $c$ is not the rightmost nonempty column of $D$. Moreover, assume that $(r,c)\in D$ is the topmost cell in column $c$ of $D$. Then, evidently, $r\ge \max\{\tilde{r}~|~(\tilde{r},c)\in \tilde{D}\in\mathrm{Min}(D)\}$. Assume that $M$ is the maximum number of cells contained in a column $\tilde{c}>c$ of $D$ and $m_c$ is the number of cells contained in column $c$ of $D$. We break the proof into three cases.
    \bigskip

    \noindent
    \textbf{Case 1:} $M>r$. Assume that $c^*$ is the leftmost nonempty column of $D$ lying strictly to the right of column $c$. Consider the diagram $\widehat{D}_{c^*}$. Note that, applying Lemma~\ref{lem:minprophelp2} (i), we have $(r,c)\in\widehat{D}_{c^*}$. We claim that $(r,c)\in D^*$ for some $\widehat{D}_{c^*}\succeq D^*\in \mathrm{Min}(D)$. To see this, applying Lemma~\ref{lem:minprophelp2} (ii), for all $\tilde{D}\preceq \widehat{D}_{c^*}$, $\tilde{c}\ge c^*$, and $\tilde{r}\ge 1$, we have that $(\tilde{r},\tilde{c})\in \tilde{D}$ if and only if $(\tilde{r},\tilde{c})\in \widehat{D}_{c^*}$. Consequently, there exists $D^*\in \mathrm{Min}(D)$ which matches $\widehat{D}_{c^*}$ in columns $\tilde{c}\ge c^*$. Considering Lemma~\ref{lem:minprophelp} (ii), it follows that for each $\tilde{r}$ satisfying $1\le \tilde{r}\le M$, there exists a column $\widehat{c}\ge c^*>c$ in $\widehat{D}_{c^*}$ for which $(\tilde{r},\widehat{c})\in\widehat{D}_{c^*}$. Thus, since $r<M$, it follows that $(r,c)\in \tilde{D}$ for all $\tilde{D}\preceq \widehat{D}_{c^*}$, i.e., $(r,c)\in D^*$, as claimed. Thus, in this case we have $r=h(D,c)=\max\{\tilde{r}~|~(\tilde{r},c)\in \tilde{D}\in\mathrm{Min}(D)\}$, as desired.
    \bigskip

    \noindent
    \textbf{Case 2:} $m_c\ge M$. Applying Lemma~\ref{lem:minprophelp} (i) to columns $\tilde{c}\ge c$, it follows that for all $\tilde{D}\in\mathrm{Min}(D)$, the cells of column $c$ occupy rows weakly below row $m_c$. As column $c$ contains exactly $m_c$ cells, it follows that for all $\tilde{D}\in\mathrm{Min}(D)$, the topmost cell in column $c$ occupies row $m_c$. Thus, $m_c=h(D,c)=\max\{\tilde{r}~|~(\tilde{r},c)\in \tilde{D}\in\mathrm{Min}(D)\}$, as desired.
    \bigskip

    \noindent
    \textbf{Case 3:} $r>M>m_c$. Applying Lemma~\ref{lem:minprophelp} (i) to columns $\tilde{c}\ge c$, it follows that for all $\tilde{D}\in\mathrm{Min}(D)$, the cells of column $c$ occupy rows weakly below row $M$. Thus, $M\ge\max\{\tilde{r}~|~(\tilde{r},c)\in \tilde{D}\in\mathrm{Min}(D)\}$. Applying Lemma~\ref{lem:minprophelp2} to $\widehat{D}_c$, it follows that there exists $D^*\in\mathrm{Min}(D)$ for which $D^*\preceq \widehat{D}_c$ and $(M,c)\in D^*$ is the topmost cell in column $c$ of $D^*$. Thus, $M=h(D,c)=\max\{\tilde{r}~|~(\tilde{r},c)\in \tilde{D}\in\mathrm{Min}(D)\}$, as desired.
\end{proof}

 The following is an immediate corollary of Proposition~\ref{prop:highestcellmin}.

\begin{theorem}\label{thm:maxmin}
    Let $D$ be a diagram with nonempty columns $\{c_1<\hdots<c_n\}$ and suppose that for $i\in [n]$, column $c_i$ of $D$ contains $n_i$ cells. Then $$\max\{\mathrm{empty}(\tilde{D})~|~\tilde{D}\in\mathrm{Min}(D)\}=\sum_{i=1}^nh(D,c_i)-n_i.$$
\end{theorem}
\begin{proof}
    Considering Proposition~\ref{prop:highestcellmin}, it remains to show that there exists $D^*\in\mathrm{Min}(D)$ in which for each nonempty column $c$, the topmost cell in column $c$ of $D^*$ occupies row $h(D,c)$. Applying Lemma~\ref{lem:minprophelp2} to $\widehat{D}_{c_1}$, we have that $\widehat{D}_{c_1}\in\mathrm{Min}(D)$ has the desired properties.
\end{proof}

Finally, to show that 
\begin{equation*}
    \mathrm{mc}(D)=|\mathrm{empty}(D)|-\max_{\widehat{D}\in\mathrm{Min}(D)}|\mathrm{empty}(\widehat{D})|.
\end{equation*} 
we require the following lemma concerning $\mathrm{empty}(D)$.

\begin{lemma}\label{lem:empty}
    For a diagram $D$, if $\tilde{D}=\mathcal{K}(D,r)$ for $r\in\mathbb{Z}_{>0}$, then $$|\mathrm{empty}(D)|-1\le |\mathrm{empty}(\tilde{D})|\le|\mathrm{empty}(D)|.$$ Moreover, if $\tilde{D}$ is formed from $D$ by moving the topmost cell in a column, then $|\mathrm{empty}(D)| - 1 = |\mathrm{empty}(\tilde{D})|$.
\end{lemma}
\begin{proof}
    Note that if $\tilde{D}=D$, then the result is immediate. So, assume that $$D\neq \tilde{D}=\mathcal{K}(D,r)=D\ldownarrow^{(r,c)}_{(\tilde{r},c)}.$$ If there exists $r^*>r$ such that $(r^*,c)\in D$, then $$\mathrm{empty}(\tilde{D})=(\mathrm{empty}(D)\backslash \{(\tilde{r},c)\})\cup\{(r,c)\},$$ i.e., $|\mathrm{empty}(\tilde{D})|=|\mathrm{empty}(D)|$; otherwise, if $(r, c)$ is the topmost cell in column $c$, then $$\mathrm{empty}(\tilde{D})=\mathrm{empty}(D)\backslash \{(\tilde{r},c)\}$$ so that $|\mathrm{empty}(\tilde{D})|=|\mathrm{empty}(D)|-1$. 
\end{proof}

\begin{theorem}~\label{thm:min}
    For a diagram $D$, $$\mathrm{mc}(D)=|\mathrm{empty}(D)|-\max_{\widehat{D}\in\mathrm{Min}(D)}|\mathrm{empty}(\widehat{D})|.$$
\end{theorem}
\begin{proof}
    Let $M=\max_{\widehat{D}\in\mathrm{Min}(D)}|\mathrm{empty}(\widehat{D})|$.
    Considering Lemma~\ref{lem:empty}, it follows that $\mathrm{mc}(D)\ge|\mathrm{empty}(D)|-M$. Thus, it remains to show that there exists $D^*\in \mathrm{Min}(D)$ for which $D^*$ can be formed from $D$ by applying a sequence of $|\mathrm{empty}(D)|-M$ Kohnert moves.

    Let $c_1<c_2<\cdots<c_n$ denote the nonempty columns of $D$. Note that, considering Proposition~\ref{prop:highestcellmin} and Lemma~\ref{lem:empty}, it suffices to show that there exists $D^*\in\mathrm{Min}(D)$ with $(h(D,c_i),c_i)\in D^*$ for $i\in [n]$ which can be formed from $D$ by applying only Kohnert moves affecting the position of the topmost cell of a column. To see this, note that by Proposition~\ref{prop:highestcellmin}, for such a $D^*\in \mathrm{Min}(D)$ we would have $|\mathrm{empty}(D^*)|=M$. Moreover, by Lemma~\ref{lem:empty}, if only Kohnert moves affecting the position of the topmost cell in a column are applied, then each such move decreases the number of empty positions by exactly one; that is, exactly $$|\mathrm{empty}(D)|-|\mathrm{empty}(D^*)|=|\mathrm{empty}(D)|-M$$ moves are applied in forming $D^*$. Now, take $\widehat{D}_{c_1}$. Considering Lemma~\ref{lem:minprophelp2}, we have that $\widehat{D}_{c_1}\in\mathrm{Min}(D)$ and $(h(D,c_i),c_i)\in\widehat{D}_{c_1}$ for $i\in [n]$. Moreover, by definition, $\widehat{D}_{c_1}$ can be formed from $D$ by applying only Kohnert moves affecting the position of the topmost cell of a column (ignore any trivial Kohnert moves in the construction). Thus, the result follows.
\end{proof}


Theorem~\ref{thm:min} in hand, it is straightforward to establish Theorem~\ref{thm:minmain}. 

\begin{proof}[Proof of Theorem~\ref{thm:minmain}]
    Recall that $D$ is assumed to have nonempty columns $c_1<c_2<\cdots<c_n$ where for $i\in [n]$, the topmost cell in column $c_i$ is in row $r_i$. In addition, we assume that for $i\in [n]$, column $c_i$ contains $n_i$ cells. Combining Theorems~\ref{thm:maxmin} and~\ref{thm:min}, we find that 
    \begin{align*}
        \mathrm{mc}(D)&=|\mathrm{empty}(D)|-\max_{\widehat{D}\in\mathrm{Min}(D)}|\mathrm{empty}(\widehat{D})|=|\mathrm{empty}(D)|-\left(\sum_{i=1}^nh(D,c_i)-n_i\right)\\
        &=\sum_{i=1}^n|\mathrm{empty}(D)\cap\{(\tilde{r},c_i)~|~\tilde{r}\ge 1\}|-\left(\sum_{i=1}^nh(D,c_i)-n_i\right)\\
        &=\sum_{i=1}^n|\mathrm{empty}(D)\cap\{(\tilde{r},c_i)~|~\tilde{r}\ge 1\}|-h(D,c_i)+n_i\\
        &=\sum_{i=1}^n[|\mathrm{empty}(D)\cap\{(\tilde{r},c_i)~|~\tilde{r}\ge 1\}|+n_i]-h(D,c_i)=\sum_{i=1}^nr_i-h(D,c_i).
    \end{align*}
    Thus, the result follows.
\end{proof}

\begin{example}
    Let $D$ be the diagram illustrated in Figure~\ref{fig:MCex} $(a)$. To compute $\mathrm{mc}(D)$, in Figure~\ref{fig:mcex} we include row labels to the left of each row, and below each column the maximum number of cells contained in a column lying strictly to the right. In addition, we label the topmost cell in each column $c$ by $h(D,c)$. Then, by Theorem~\ref{thm:minmain}, for each labeled cell in Figure~\ref{fig:mcex}, we subtract the label from the row value and sum up the resulting values. Here, we get 
    $$\mathrm{mc}(D)=(6-2)+(6-2)+(6-2)+(4-2)+(5-3)+(7-3)+(7-4)+(2-2)=23.$$
    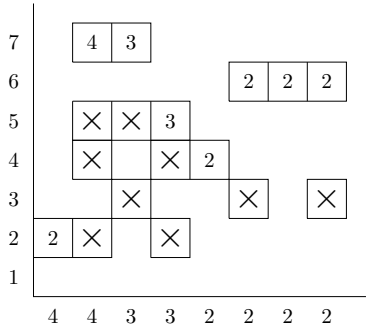
\begin{figure}[H]
        \centering
        $$\scalebox{0.8}{\begin{tikzpicture}[scale=0.65]
    \node at (0.5, 1.5) {$2$};
        \node at (1.5, 1.5) {$\bigtimes$};
        \node at (2.5, 2.5) {$\bigtimes$};
        \node at (5.5, 2.5) {$\bigtimes$};
        \node at (1.5, 3.5) {$\bigtimes$};
        \node at (3.5, 3.5) {$\bigtimes$};
        \node at (1.5, 4.5) {$\bigtimes$};
        \node at (2.5, 4.5) {$\bigtimes$};
        \node at (1.5, 6.5) {$4$};
        \node at (5.5, 5.5) {$2$};
        \node at (6.5, 5.5) {$2$};
        \node at (-0.5, 0.5) {$1$};
        \node at (-0.5, 1.5) {$2$};
        \node at (-0.5, 2.5) {$3$};
        \node at (-0.5, 3.5) {$4$};
        \node at (-0.5, 4.5) {$5$};
        \node at (-0.5, 5.5) {$6$};
        \node at (-0.5, 6.5) {$7$};

        \node at (7.5, 2.5) {$\bigtimes$};
        \node at (7.5, 5.5) {$2$};

        \node at (4.5, 3.5) {$2$};

        \node at (3.5, 4.5) {$3$};

        \node at (3.5, 1.5) {$\bigtimes$};

        \node at (2.5, 6.5) {$3$};

        \draw (0,7.5)--(0,0)--(8.5,0);
        
        \draw (0,1)--(2,1)--(2,2)--(0,2)--(0,1);
        \draw (1,1)--(1,2);
        \draw (3,1)--(4,1)--(4,2)--(3,2)--(3,1);
        
        \draw (2,2)--(3,2)--(3,3)--(2,3)--(2,2);
        \draw (5,2)--(6,2)--(6,3)--(5,3)--(5,2);
        \draw (7,2)--(8,2)--(8,3)--(7,3)--(7,2);

        \draw (1,3)--(2,3)--(2,4)--(1,4)--(1,3);

        \draw (3,3)--(5,3)--(5,4)--(3,4)--(3,3);
        \draw (4,3)--(4,4);

        \draw (1,4)--(4,4)--(4,5)--(1,5);
        \draw (1,4)--(1,5);
        \draw (2,4)--(2,5);
        \draw (3,4)--(3,5);
        \draw (5,5)--(8,5)--(8,6)--(5,6)--(5,5);
        \draw (6,5)--(6,6);
        \draw (7,5)--(7,6);
        \draw (1,6)--(3,6)--(3,7)--(1,7)--(1,6);
        \draw (2,6)--(2,7);
        \node at (0.5,-0.5) {4};
        \node at (1.5,-0.5) {4};
        \node at (2.5,-0.5) {3};
        \node at (3.5,-0.5) {3};
        \node at (4.5,-0.5) {2};
        \node at (5.5,-0.5) {2};
        \node at (6.5,-0.5) {2};
        \node at (7.5,-0.5) {2};
\end{tikzpicture}}$$
        \caption{Example computing $\mathrm{mc}(D)$}
        \label{fig:mcex}
    \end{figure}
\end{example}

Considering the work above, we are led to the following algorithm for solving the minimum Kohnert move puzzle.

\begin{tcolorbox}[breakable, enhanced]
\centerline{\textbf{\underline{Solution to Minimum Kohnert Move Puzzle}}}
\bigskip

Given a diagram $D$ with nonempty columns $c_1<c_2<\cdots<c_n$, suppose that for $i\in [n]$, 
\begin{itemize}
    \item the cell $(r_i,c_i)\in D$ is the topmost in column $c_i$ of $D$,
    \item column $c_i$ of $D$ contains $n_i$ cells, and
    \item $M_i$ denotes the maximum number of cells contained in a column strictly to the right of column $c_i$ in $D$.
\end{itemize}
Then proceed as follows. For $i\in [n]$ in decreasing order:
\begin{enumerate}
    \item[] Step $i$: If $r_i>\max\{n_i,M_i\}$, apply sequentially one Kohnert move at rows $r_i$ down to $\max\{n_i,M_i\}+1$; otherwise, do nothing.
\end{enumerate}
Note that to solve the puzzle we apply Kohnert moves to lower the topmost cell in a column as far down as possible, working in columns from right to left.

\end{tcolorbox}

\printbibliography

\end{document}